\documentclass[12pt,a4paper]{amsart}

\usepackage[english]{babel}
\usepackage[utf8x]{inputenc}
\usepackage[T1]{fontenc}

\usepackage[a4paper,top=3cm,bottom=2cm,left=3cm,right=3cm,marginparwidth=1.75cm]{geometry}

\usepackage{amsmath, amssymb, amsthm}
\usepackage{graphicx} 
\usepackage{bmpsize}
\usepackage{tikz,pgfplots}
\usepackage{enumitem}

\newtheorem{theorem}{Theorem}
\newtheorem*{theorem*}{Theorem}
\newtheorem*{proposition*}{Proposition}
\newtheorem{proposition}[theorem]{Proposition}
\newtheorem{lemma}[theorem]{Lemma}
\newtheorem{corollary}[theorem]{Corollary}
\newtheorem*{corollary*}{Corollary}
\newtheorem*{claim*}{Claim}
\newtheorem{definition}{Definition}
\theoremstyle{remark}
\newtheorem{remark}{Remark}
\newtheorem*{remark*}{Remark}
\pgfplotsset{compat=newest}

\title{Existence Results for the Nonlinear Hodge Minimal Surface Energy}
\author{Daniel Agress}
\address{Department of Mathematics, University of California Irvine, 92617} 
\email{dagress@uci.edu}
\date{3/5/2018}

\begin{document}

\begin{abstract}
Given a compact Riemannian manifold $(M^n,g)$ and a fixed cohomology class, $[\alpha^*] \in H^k(M)$, we consider the existence of a minimizer $\alpha \in [\alpha^*]$ of the generalized minimal surface energy $\int_M \sqrt{1+|\alpha|^2} dV_g$. When $k = 1$, we prove the existence of unique minimizers for every cohomology class $[\alpha^*]$. Next, when $k > 1$, we construct examples of singular solutions for finite cohomology class $[\alpha^*] \in H^k(S^k \times S^k,g)$, where $g$ is conformal to the standard metric on $S^k \times S^k$. Additionally, we show that when $k=2$, these singular solutions are also solutions to the Born Infeld equation.
\end{abstract}

\maketitle

\section{Introduction} \label{intro}

Given a compact Riemannian manifold $(M^n,g)$, $\alpha \in \Lambda^k(M)$, Hodge theory studies the critical points of the energy
$$E(\alpha) = \int_M |\alpha|^2 dV_g$$
within a fixed cohomology class $[\alpha^*]$. In  \cite{sibnerExistence2},\cite{sibnerExistence1}, \cite{sibnerExample1}, and \cite{sibnerExample2}, Robert and Lesley Sibner studied the nonlinear Hodge problem, which, given a function $f : \mathbb{R}^+ \rightarrow \mathbb{R}$, studies the existence of critical points of the nonlinear Hodge energy
$$E_f(\alpha) = \int_M f(|\alpha|) dV_g$$
over a fixed cohomology class. One important case of a nonlinear Hodge energy, which we will call the generalized minimal surface energy (GMS), is given by
$$ E_{GMS} (\alpha) = \int_M \sqrt{1+|\alpha|^2} dV_g$$
This energy finds applications in several diverse settings. 
\begin{enumerate}
\item When $\alpha = du$ is an exact one form then the minimal surface energy is simply the area of the graph $\{(x,u(x))\} \subseteq M \times \mathbb{R}$. Thus, minimizers of the GMS energy correspond to graphical minimal surfaces in $M \times \mathbb{R}$. If $\alpha$ is not exact, minimizers of the GMS energy correspond multivalued minimal graphs in $M \times \mathbb{R}$ whose equivariant gradient is the one form $\alpha$. These graphs can be lifted to entire minimal surfaces over the universal cover $\tilde{M}$ whose gradient is equivariant over the group of deck transformations of $M$. 
\item In the case $\alpha = du \in \Lambda^1(M)$, the energy $\int_M |du| dV_g$ is known as the total variation, $TV(du)$. Functions $u \in L^1(M)$ with $TV(du) < \infty$ are known as the functions of bounded variation. The space of BV functions has been extensively studied for its relevance to minimal hypersurfaces (see Giusti \cite{giusti}) and in image processing (see \cite{ambrosio}). As we will discuss in Section \ref{limits}, many of these results can be generalized to the case where $\alpha \in \Lambda^1(M)$ is not exact. 

The $t-GMS$ energy, $E^t_{GMS}(\alpha) := \int_M \sqrt{t^{-2} + |\alpha|^2}dV_g$ provides a regularization of the total variation. We will show in Theorem \ref{thm2} below that as $t \rightarrow \infty$, $E^t_{GMS}$ $\Gamma$-converges to the total variation and that minimizers converge weakly to a minimizer of the total variation. 

\item In 1934, in \cite{born}, Born and Infeld introduced a nonlinear theory of electromagnetism with the Lagrangian  $\int_M \sqrt{\det(g-F)}$. Here, $g$ is the metric and $F \in \Lambda^2(M)$ is the electromagnetic field. This Lagrangian has relevance to contemporary physics in the context of electromagnetic fields in string theory. When $M$ is a three dimensional Riemannian manifold, the Born Infeld energy reduces to the GMS energy, $\int_M \sqrt{1+|F|^2}dV_g$. Thus, GMS $2$-forms are magnetostatic solutions of the Born Infeld theory of electromagnetism. We will see in section \ref{BI} that in special cases, minimizers of the GMS energy also minimize the full Born Infeld energy in four dimensions. 
\end{enumerate}

As mentioned earlier, in \cite{sibnerExistence2} and \cite{sibnerExistence1} Sibner and Sibner studied the general nonlinear Hodge problem. Their results, as we will describe in Section \ref{sibnerexistence}, show that for small enough cohomology classes there exists a smooth minimizer of the minimal surface energy. However, they leave open the question of whether solutions become singular for large cohomology class. One simple case in which a solution exists for every cohomology class is when the harmonic form of the cohomology class is parallel, in which case it is automatically a GMS solution. In particular, for $M$ with non-negative curvature, there is a GMS solution in every cohomology class. However, for more general compact Riemannian manifolds, the question of the existence of solutions for large cohomology class was open. Our results resolve this question by showing that for $k=1$, a smooth solution exists in every cohomology class, whereas for $k>1$, singular solutions exist for finite cohomology class. 
\begin{theorem}\label{thm1}
Given a compact Riemannian manifold $(M^n,g)$ and a fixed cohomology class $[\alpha^*] \in H^1(M)$, for every $t \in \mathbb{R}$, there exists a unique $\alpha_t \in [t\alpha^*]$ which minimizes the generalized minimal surface energy 
$$ E_{GMS} (\alpha) = \int_M \sqrt{1+|\alpha|^2} dV_g$$
over the cohomology class $[t\alpha^*]$.
\end{theorem}
We can also study the rescaled minimizers $\beta_t = t^{-1} \alpha_t \in [\alpha^*]$. We will see that $\beta_t$ minimizes the rescaled $t-GMS$ energy, $E^t_{GMS}(\alpha) = \int_M \sqrt{t^{-2} + |\alpha|^2} dV_g$. With the aim of understanding the minimizers of the BV energy, we can study the limit as $t \rightarrow 0$ and $t \rightarrow \infty$. We prove

\begin{theorem}\label{thm2}
Given a compact Riemannian manifold $(M^n,g)$ and a cohomology class $[\alpha^*] \in H^1(M)$, let $\beta_t$ be minimizers of the $t$-GMS energy
$$E^t_{GMS}(\alpha) = \int_M \sqrt{t^{-2} + |\alpha|^2} dV_g$$
over the cohomology class $[\alpha^*]$. Let $\alpha_H$ be the harmonic form in the class $[\alpha^*]$. Then,
\begin{enumerate}
\item As $t \rightarrow 0$, $\beta_t \xrightarrow{C^{\infty}} \alpha_H$.
\item If $t_k \rightarrow \infty$, $\exists$ a subsequence $t_{k_n}$ and a BV minimizer $\alpha_0 \in [\alpha^*]$ such that  $\beta_{t_{k_n}} \xrightarrow{WBV} \alpha_0$.
\end{enumerate}
\end{theorem}
\begin{remark*} 
Here, the $WBV$ (weak BV) convergence is as described in section \ref{limits}. In fact, we prove that the $t-GMS$ energy $\Gamma$-converges (see \cite{braides}) to the total variation as $t \rightarrow \infty$. From Theorem \ref{thm2}, we see that the $t-GMS$ minimizers provide a smooth one parameter family of forms which link the harmonic form and a BV minimizer of the cohomology class. 
\end{remark*}

We next show that although a solution exists for every cohomology class in $H^1(M)$, for general $k$-forms this is not true. There are large cohomology classes in $H^k(M)$ where no GMS solution exists. Thus, the small cohomology existence result of \cite{sibnerExistence2} is optimal for higher degree differential forms. We prove this by constructing explicit families of solutions which become singular for finite cohomology class.
\begin{theorem}\label{thm3}
There exists a metric on $S^k \times S^k$, $\Lambda > 0$, and a cohomology class $[\alpha^*] \in H^k(S^k \times S^k)$ such that a smooth GMS solution exists in the cohomology class $[t \alpha^*]$ iff $|t| < \Lambda$.
\end{theorem}
We also show that in the case $k=2$, these explicit solutions minimize the Born Infeld energy. Using Theorem \ref{thm3}, we then show that
\begin{theorem}
There exists a metric on $S^2 \times S^2$ and a cohomology class $[\alpha^*] \in H^2(S^2 \times S^2)$ which contains no smooth minimizer of the Born Infeld energy. 
\end{theorem}

\textbf{Acknowledgements:} The author would like to thank his advisors Patrick Guidotti and Jeffrey Streets, as well as Richard Schoen for their insight and advice.

\section{Preliminaries}\label{prelim}
As described above, we will be considering a compact Riemannian manifold $(M^n,g)$. $[\alpha^*]$ will be a fixed cohomology class, and we will be studying critical points of $E_{GMS}(\alpha) = \int_M \sqrt{1+|\alpha|^2}dV_g$ for $\alpha \in [\alpha^*]$. We will define $[\alpha^*]_{C^{\infty}}$ and $[\alpha^*]_{W^{1,2}}$ as the $C^{\infty}$ and $W^{1,2}$ forms in $[\alpha^*]$, respectively. We can extend the $GMS$-energy to $\alpha \in [\alpha^*]_{W^{1,2}}$ by defining 
$$E_{GMS}(\alpha) = \inf_{\left\{ \{\alpha_n\} \subseteq [\alpha^*]_{C^{\infty}}\big| \alpha_n \xrightarrow{W^{1,2}}\alpha \right\}} \liminf_{n \rightarrow \infty } E_{GMS}(\alpha_n).$$
\begin{lemma}
A smooth form $\alpha \in [\alpha^*]$ is a minimizer of the minimal surface energy over $[\alpha^*]$ iff it satisfies
\begin{equation*}
d^* \left( \frac{\alpha}{\sqrt{1+|\alpha|^2}} \right) = 0.
\end{equation*}
Moreover, if such a minimizer exists, it is the unique minimizer in $[\alpha^*]_{W^{1,2}}$. 
\end{lemma}
\begin{proof}
We begin by calculating the first variation of energy to obtain the Euler Lagrange equation. We take variations of the form $\alpha(t) = \alpha_0 + td\psi$, which fixes the cohomology class. We calculate
\begin{align*}
\frac{d}{dt} \bigg|_{t=0} \int_M \sqrt{1+|\alpha(t)|^2} dV_g & = 
\int_M \frac{\frac{1}{2} \frac{d}{dt} |\alpha(t)|^2}{\sqrt{1+|\alpha_0|^2}} dV_g \\
& = \left< \frac{\alpha_0}{\sqrt{1+|\alpha_0|^2}} , d\psi \right>.
\end{align*}
Integrating by parts, we obtain the GMS equation
\begin{equation*}
d^* \left( \frac{\alpha_0}{\sqrt{1+|\alpha_0|^2}} \right) = 0.
\end{equation*}
Thus, every minimizer will satisfy the given Euler Lagrange equation. To show the converse, we calculate the second variation and show that the energy is strictly convex. We again consider variations of the form $\alpha(t) = \alpha_0 + td\psi$. 
\begin{align*}
\frac{d^2}{dt^2} \bigg|_{t=t_0} \int_M \sqrt{1+|\alpha(t)|^2} dV_g 
& =  \int_M \left( \frac{\frac{1}{2}\frac{d^2}{dt^2}|\alpha(t)|^2}{\sqrt{1+|\alpha(t_0)|^2}} - \frac{\big|\frac{1}{2}\frac{d}{dt} |\alpha(t)|^2\big|^2}{(1+|\alpha(t_0)|^2)^{3/2}}\right) dV_g \\
& =  \int_M \left( \frac{|d\psi|^2}{\sqrt{1+|\alpha(t_0)|^2}} - \frac{\big| (\alpha(t_0),d\psi)\big|^2}{(1+|\alpha(t_0)|^2)^{3/2}}\right) dV_g \\
& \geq \int_M \frac{|d\psi|^2}{(1+|\alpha(t_0)|^2)^{3/2}} dV_g.
\end{align*}
The last inequality follows from the Cauchy Schwartz inequality. As the second variation is strictly positive, we obtain that the functional is strictly convex and any critical point must be its unique minimum over $[\alpha^*]_{C^{\infty}}$.

We now assume that $\alpha$ is a smooth minimizer and turn to proving taht $\alpha$ is a unique minimizer over $[\alpha^*]_{W^{1,2}}$. By definition of the $GMS$-energy on forms in $W^{1,2}$ we see that 
$$\inf_{\beta \in [\alpha^*]_{W^{1,2}}} E_{GMS}(\beta) = E_{GMS}(\alpha).$$
We now assume that $\beta \in [\alpha^*]_{W^{1,2}}$ is a minimizer, i.e. $E_{GMS}(\beta) = E_{GMS}(\alpha)$. By the definition of the $GMS$-energy, we can take a sequence of smooth form $\beta_n$ such that 
$$\beta_n \xrightarrow{W^{1,2}} \beta \quad \text{ and } \quad E_{GMS}(\beta_n) \rightarrow E_{GMS}(\beta).$$
Setting $\gamma_n = \beta_n - \alpha$, we note that $\forall t \in [0,1]$, 
$$ \alpha + t \gamma_n \xrightarrow{W^{1,2}} \alpha + t \gamma \quad \text{ and } \quad E_{GMS} (\alpha + t\gamma_n) \rightarrow E_{GMS} (\alpha + t \gamma). $$
By the convexity of the $GMS$-energy for smooth forms, we know that 
$$\forall t \in [0,1],\quad \forall n\in\mathbb{N}, \quad E_{GMS}(\alpha) \leq E_{GMS}(\alpha + t\gamma_n) \leq E_{GMS}(\alpha + \gamma_n).$$
Taking the limit as $n \rightarrow \infty$, we find that 
$$\forall t \in [0,1], \quad E_{GMS}(\alpha + t\gamma) \equiv E_{GMS}(\alpha).$$
In particular, given that $E_{GMS}(\alpha + t\gamma_n)$ are monotone increasing functions of $t$, we find that
$$\lim_{n \rightarrow \infty} \frac{d^2}{dt^2} E_{GMS}(\alpha+t\gamma_n) = 0.$$
On the other hand, we know from the calculation of the second variation that 
\begin{align*}
\frac{d^2}{dt^2} E_{GMS}(\alpha + t\gamma_n) &
\geq \int_M \frac{|\gamma_n|^2}{(1+|\alpha|^2)^{3/2}} dV_g \\
& \geq c||\gamma_n||_{L^2}.
\end{align*}
Putting these two equation together, we find that 
$$\lim_{n \rightarrow \infty} \gamma_n = 0$$
and therefore $\alpha = \beta$.
\end{proof}

\subsection{Existence for small cohomology class}\label{sibnerexistence}
We next describe the results of Sibner \cite{sibnerExistence1} which prove existence of GMS solutions for small cohomology class. We begin by introducing some terminology from \cite{sibnerExistence1}. A nonlinear Hodge problem is one of the form
$$ d^* \left( \rho(|\alpha|)\alpha \right) = 0.$$
Here, $\rho: \mathbb{R}^+ \rightarrow \mathbb{R}$ is a smooth function.
\begin{definition}
The function $\rho$ is called \emph{admissible} if there exist constants $c$ and $k(c)$ such that for $0 < x < c$ we have 
\begin{enumerate}
\item $\frac{1}{k(c)} < \rho(x) < k(c)$
\item $\frac{1}{k(c)} < \rho(x) + \rho'(x)x < k(c)$
\end{enumerate} 
The supremum $Q_{\rho}$ of such values $c$ such that there exists such a $k(c)$ is called the \emph{sonic value} of $\rho$. If the $Q_{\rho} = \infty$ and $k$ can be chosen independent of $c$ then $\rho$ is called \emph{regular}. 
\end{definition}
\begin{remark}
We note that by setting $\rho_{GMS}(|\alpha|) = \frac{1}{\sqrt{1+|\alpha|^2}}$, the minimal surface equation is of the form 
$$ d^* \left( \rho_{GMS}(|\alpha|)\alpha \right) = 0.$$
Furthermore, $\rho_{GMS}$ is admissible with $Q_{\rho_{GMS}} = \infty$. However, $\rho$ is not regular.
\end{remark}
These conditions serve as ellipticity conditions on the PDE, and allowed the Sibners to show existence for small cohomology class. 
\begin{theorem}\label{sibnerthm}
(\cite{sibnerExistence2} Thm. 1)Given a Riemannian manifold $(M,g)$, a cohomology class $[\alpha^*]$, and $\rho$ an admissible function with sonic value $Q_{\rho}$, then $\exists T>0$ such that
\begin{enumerate}
\item For all $t \in [0,T)$, $\exists \alpha_t \in [t\alpha_0]$ such that $d^* \left( \rho(|\alpha_t|) \alpha_t \right) = 0$. 
\item $\lim_{t \rightarrow T} \sup_{x\in M}{|\alpha_t|} = Q_{\rho}$. 
\item $\alpha_t$ depends continuously on $t$ in the topology of uniform convergence.
\end{enumerate}
Furthermore, if $\rho$ is regular, then a solution exists for every cohomology class.
\end{theorem}

With respect to the GMS equation, we find that for any cohomology class $[\alpha^*]$ there exists $T>0$ such that
\begin{enumerate}
\item For $t < T$, there is a unique solution of the minimal surface equation in the class $[t\alpha^*]$
\item $\lim_{t \rightarrow T} \sup_{x\in M}{|\alpha_t|} = \infty$. 
\end{enumerate}
However, as $\rho_{GMS}$ is not regular, it remains unclear whether $T = \infty$ or whether singularities can form in finite cohomology class. In the next section, we will show that for $[\alpha^*] \in H^1(M)$ the answer is that a solution exists in every cohomology class. In contrast, in section \ref{examples}, we construct counterexamples where singularities occur for finite cohomology class when $k>1$.

\section{Proof of Theorem 1}\label{mainproof}

To prove the result, we will require the following two theorems from the theory of minimal surfaces. For the rest of the section we will write $B_{\rho}(p)$ for the ball of radius $\rho$ in the manifold $M$, while $\bar{B}_{\rho}(p)$ will denote the ball of radius $\rho$ in the manifold $M \times \mathbb{R}$. 

\begin{theorem}\label{lemma1}\label{volumethm}
Given a compact manifold $M^n$, $\exists \epsilon$ and $\rho_0$ such that $\forall \rho < \rho_0$, if $\Sigma \subset M \times \mathbb{R}$ is a minimal surface, then $\forall x_0 \in \Sigma$ we have $|\bar{B}_{\rho}(x_0) \cap \Sigma| > \epsilon \rho^n$.
\end{theorem}

\begin{proof}
We remark that the statement is a simple corollary of the monotonicity formula for minimal surfaces. To begin, using the Nash embedding theorem, we embed our manifold $M \times \mathbb{R}$ into Euclidean space $\mathbb{R}^l$. As $\Sigma$ is a minimal surface, we find that there exists a constant $\Lambda$ such that $|H_{\Sigma}| < \Lambda$. Here, $H_{\Sigma}$ is the second fundamental form of $\Sigma$ as an embedded surface in $\mathbb{R}^l$. We then apply the monotonicity formula to $\Sigma \subset \mathbb{R}^l$, found in Simon, 17.6, \cite{simon}, which states that if $0 < \sigma < \rho < R$, then
$$e^{\Lambda \sigma} \sigma^{-n} |\Sigma \cap \bar{B}_{\sigma}(p)| \leq e^{\Lambda  \rho} \rho^{-n} |\Sigma \cap \bar{B}_{\rho}(p)|$$
We note that as $\sigma \rightarrow 0$, the left hand side converges to $\omega_n$, the volume of the $n$-dimensional unit ball. Furthermore, given that $\rho < R$, $e^{\Lambda \rho } \leq e^{\Lambda R}$. Therefore, we find that
$$  \omega_n e^{-\Lambda R} \rho^n \leq |\Sigma \cap \bar{B}_{\rho}(p)|.$$
\end{proof}
We next quote a theorem from Spruck \cite{spruck} which gives $C^0$ gradient estimates on minimal surfaces in $M \times \mathbb{R}$. These estimates were originally proved for Euclidean space by Bombieri in \cite{bombieri} for low dimension and Simon in \cite{simonEstimates} for high dimension. In the following theorem, Spruck generalizes their results to the case of $M \times \mathbb{R}$. 
\begin{theorem}\label{spruckthm}
(\cite{spruck} Thm 1.1) If $(x,u(x))$ is a minimal graph in $B_{\rho}(p) \times \mathbb{R}$ and $u \geq 0$, then 
$$\sqrt{1+|du(p)|^2} \leq 32 \max \left( 1,\left( \frac{u(p)}{\rho} \right)^2 \right) \exp^{16 C u(p)} \exp^{16 C \left( \frac{u(p)}{\rho} \right)^2} $$
where $C$ is a constant which depends only on the sectional curvatures of $M$ and an upper bound for $\Delta d_p^2$ on $B_{\rho}$, where $d_p$ is the distance function from $p$.
\end{theorem}

\begin{proof}[Proof of Theorem \ref{thm1}]
We recall that by Theorem \ref{sibnerthm}, it suffices to show 
$$\forall T < \infty, \quad \exists C \text{ s.t. for } t<T, \quad \sup_{x \in M} |\alpha_t| < C,$$
where $\alpha_t \in [t\alpha^*]$ is the GMS solution.

Let $p \in M$. We take $\rho < \rho_0$ (from Theorem \ref{volumethm}) such that $2\rho$ is less than the injectivity radius of $M$. Then $B_{2\rho}(p) \subset M$ is simply connected. As $\alpha_t$ is a closed $1$-form, there exists a function $u_t:B_{2\rho}(p) \rightarrow \mathbb{R}$ such that $\alpha_t \big|_{B_{2\rho}(p)} = du_t$. Furthermore, the GMS equation for $\alpha_t$ states that 
$$d^* \left( \frac{du_t}{\sqrt{1+|du_t|^2}} \right) = 0.$$
Thus $u_t$ defines a minimal graph over $B_{2\rho}(p)$. We define $\Sigma := \{(x,u(x) | x \in B_{2\rho}(p) \}$. We set 
$$\kappa := E_{GMS}(T\alpha^*) = \int_M \sqrt{1+|T\alpha^*|^2}dV_g.$$
Because $\alpha_t$ minimizes the GMS energy in its cohomology class,  
\begin{align*}
\int_{B_{2\rho}(p)} \sqrt{1+|du_t|^2} dV_g 
& \leq \int_M \sqrt{1+|\alpha_t|^2} dV_g \\
& \leq \int_M \sqrt{1+|t\alpha^*|^2} dV_g < \kappa.
\end{align*}
We now show that this implies that 
$$\sup_{x \in B_{\rho}(p)} u_t(x) - \inf_{x \in B_{\rho}(p)} u_t(x) \leq \rho \left(1+\frac{\kappa}{\epsilon \rho^n} \right) .$$
Here, $\epsilon$ is taken from Theorem \ref{volumethm}. We argue by contradiction. Assume that there exist $x,y \in B_{\rho}(p)$ such that $u(x) - u(y) > \rho(1+\frac{\kappa}{\epsilon \rho^n})$. As $u$ is a smooth function, we know that there exists a curve connecting $(x, u(x))$ and $(y, u(y))$ lying in $\Sigma \cap (B_{\rho}(p) \times \mathbb{R})$. We take $m := \lfloor \frac{u(x) - u(y)}{\rho} \rfloor$ points along the curve, given by $\{ (x_1,u(x_1), \ldots, (x_n,u(x_n)) \}$ such that $u(x_k) - u(x_{k-1}) = \rho$ for $1 < k \leq m$. We note that by assumption, $m > \frac{\kappa}{\epsilon \rho^n}$. Then, we take the $m$ balls $\bar{B}_k := \bar{B}_{\rho}((x_k,u(x_k))$; see figure \ref{fig:mingraph}.
\begin{figure}[ht]
\centering
\caption{}
\begin{tikzpicture}
\draw (-2,-2) -- (2,-2) node[anchor=north]{$B_{\rho} \subset M$};
\draw (-2,-2) -- (-2,2) node[anchor=east]{$\mathbb{R}$};
\draw (-2,-2) .. controls (-1,1) and (1,-1) .. (2,2)
	node[pos = .1]{$\bullet$}
	node[pos = .5]{$\bullet$}
	node[pos = .9]{$\bullet$}
	node[pos=.1,circle,inner sep=3pt,draw,label=0:$\bar{B}_{k-1}$]{\phantom{text}}
	node[pos=.5,circle,inner sep=3pt,draw,label=0:$\bar{B}_k$]{\phantom{text}}
	node[pos=.9,circle,inner sep=3pt,draw,label=0:$\bar{B}_{k+1}$]{\phantom{text}}
	node[anchor=west]{$\Sigma$}; 
\end{tikzpicture}
\label{fig:mingraph}
\end{figure}
We note that, by construction, the $\bar{B}_k$ are mutually disjoint and each $\bar{B}_k$ lies entirely in $B_{2\rho}(p) \times \mathbb{R}$. Furthermore, by Theorem \ref{volumethm}, we know that $|\Sigma \cap \bar{B}_k| \geq \epsilon \rho^n$. Thus, using the definition of $m$,
$$| \Sigma \cap (B_{2\rho}\times \mathbb{R}) | \geq \cup_{k=1}^m |\Sigma \cap \bar{B}_k| \geq m\epsilon \rho^n > \kappa,$$
a contradiction.

Now, by shifting $u_t$ by a constant, we obtain a new function $\tilde{u}_t$ on $B_{\rho}$ such that $\tilde{u}_t \geq 0$ and $\sup \tilde{u}_t < \rho \left( 1 + \frac{\kappa}{\epsilon \rho^n} \right)$. Then, applying Theorem \ref{spruckthm}, we find that there is a constant, $C$ independent of $p$ such that 
$$ \sqrt{1+|d\tilde{u}_t(p)|^2} \leq C.$$
\end{proof}

%

\section{Limits of GMS Solutions}\label{limits}

We now prove Theorem \ref{thm2}. Given a cohomology class $\alpha^*$, we have, by Theorem \ref{thm1}, a 1-parameter family of solutions to the GMS equation $\alpha_t \in [t\alpha^*]$. We rescale the GMS solutions by defining 
$$ \beta_t = t^{-1} \alpha_t. $$
We note that $\beta_t \in [\alpha^*]$ for all $t$. Furthermore, 
$$ GMS(\alpha_t) = \int_M \sqrt{1+\alpha_t^2}dV_g = \frac{1}{t} \int_M \sqrt{t^{-2} + \beta_t^2} dV_g.$$
We now define the rescaled $t-GMS$ energy to be
$$ E^t_{GMS}(\alpha) = \int_M \sqrt{t^{-2} + |\alpha|^2} dV_g. $$
Thus, $\beta_t \in [\alpha^*]$ minimizes the $t-GMS$ energy iff $\alpha_t$ minimizes the GMS energy in $[t\alpha^*]$. We can study the limiting behavior as $t \rightarrow 0$ and $t \rightarrow \infty$. We begin by showing the first part of Theorem \ref{thm2}.
\begin{proposition}
As $t \rightarrow 0$, $\beta_t \xrightarrow{C^{\infty}} \alpha_H$, where $\alpha_H$ is the harmonic representative of $[\alpha^*]$. 
\end{proposition}

\begin{proof}
A $t-GMS$ solution satisfies
$$ d^* \left( \frac{\beta_t}{\sqrt{t^{-2} + |\beta_t|^2}} \right) = 0.$$
We evaluate to find
\begin{align*}
d^* \left( \frac{\beta_t}{\sqrt{t^{-2} + |\beta_t|^2}} \right) 
& = \frac{d^* \beta_t}{\sqrt{t^{-2} + |\beta_t|^2}} - \frac{*(\frac{1}{2}\nabla|\beta_t|^2 \wedge *\beta_t)}{(t^{-2} + |\beta_t|^2)^{3/2}} = 0
\end{align*}
We multiply the equation by $t^2(t^{-2} + |\beta_t|^2)^{(3/2)}$ to obtain
\begin{align*}
d^*\beta_t + t^2 \left( |\beta_t|^2 d^* \beta_t - * \frac{1}{2} \nabla |\beta_t|^2 \wedge * \beta_t \right) = 0
\end{align*}
We rewrite the equation in local coordinates, setting $\beta_t = \alpha^* + df_t$.
\begin{align*}
\nabla^i \nabla_i f_t & + t^2|\beta_t|^2 \nabla^i \nabla_i f_t - t^2(\beta_t)_i (\beta_t)_j \nabla^i \nabla_j f_t \\
& = -\nabla^i (\alpha^*)_i - t^2|\beta_t|^2 \nabla^i (\alpha^*)_i + t^2 (\beta_t)_i (\beta_t)_j \nabla^i (\alpha^*)_j.
\end{align*}
We now recognize that $t \beta_t = \alpha_t$. Thus we obtain
\begin{align*}
\nabla^i \nabla_i f_t & + |\alpha_t|^2 \nabla^i \nabla_i f_t - (\alpha_t)_i (\alpha_t)_j \nabla^i \nabla_j f_t \\
& = -\nabla^i (\alpha^*)_i - |\alpha_t|^2 \nabla^i (\alpha^*)_i + (\alpha_t)_i (\alpha_t)_j \nabla^i (\alpha^*)_j.
\end{align*}
We now recall from Statement 3 of Theorem \ref{sibnerthm} that as $t \rightarrow 0$, $\alpha_t$ converges uniformly to $0$. Thus, by taking $T$ small enough, we know that for any $k, \gamma$, there is a constant $C$ such that for any $t \in [0,T]$, 
$$||\alpha_t||_{C^{k,\gamma}} < C.$$
Thus, for $t \in [0,T]$, we obtain a uniform $C^{k,\gamma}$ bound on the coefficients in the PDE. Thus, by Schauder estimates, we obtain a uniform bound 
$$||f_t||_{C^{k+2,\gamma}} < C.$$ 
Applying Arzela-Ascoli, for any $t_n \rightarrow 0$, we find a convergent subsequence in $C^{2,\gamma}$ with limit $f_0$. Then, taking the limit of $\alpha_t \rightarrow 0$ in the PDE, we find that $f_0$ satisfies
$$\Delta f_0 = -d^* (\alpha^*).$$
Thus, $\alpha^* + df_0 = \alpha_H$. As every sequence must have a  subsequence which converges to $f_0$, we obtain that $f_t \rightarrow f_0$. 
\end{proof}

We now turn to the limiting behavior as $t \rightarrow \infty$. In general, we cannot expect convergence of the sequence in a classical sense. However, we can define weak convergence in two equivalent ways: in the sense of currents or in the sense of BV functions. In this paper I will discuss the convergence in the sense of BV functions. 

We will begin by defining the space of functions of bounded variation. The total variation of an $L^1$ function is defined to be
$$TV(df) = \sup_{\beta \in \Lambda^1(M),|\beta|<1} \int_M <f, d^* \beta> dV_g.$$
This allows us to define the space of bounded variation functions.
$$BV(M) = \{ f \in L^1(M) | TV(df) < \infty \}.$$
This space can be given both a strong and a weak topology. The strong topology is given by the norm
$$||f||_{BV} = ||f||_{L^1} + TV(df).$$
It can also be given a weak topology. For $f_k, f_0 \in BV(M)$, we say that
$$f_k \xrightarrow{WBV} f_0 \text{ if } f_k \xrightarrow{L^1} f_0 \text{ and } TV(df_k) \rightarrow TV(df).$$
We now refer to \cite{giusti} for the following two properties of $BV(M)$.
\begin{enumerate}
\item Compactness. \label{compactness}
$$\forall \{f_k\} \in BV, ||f_k||_{BV} < C, \exists f_{k_n} \text{ and } f_0 \text{ s.t. } f_{k_n} \xrightarrow{L^1} f_0.$$
\item Lower Semicontinuity. \label{lsc}
$$\text{If } \{f_k\} \in BV \text{ and } f_k \xrightarrow{L^1} f \text{ then } TV(df) \leq \liminf TV(df_k). $$
\item Density of smooth functions in the weak topology. \label{density}
$$\forall f \in BV(M), \quad \exists f_k \in C^{\infty}(M) \text{ s.t. } f_k \xrightarrow{WBV} f.$$
\end{enumerate}
Because we are in the nontrivial $[\alpha^*]$ cohomology class, we will need to introduce the modified $TV_{\alpha^*}$ energy.  To begin, we find a partition of unity $(U_i, \phi_i)$. Then for a one closed form $\alpha^*$, we can find $g_i$ such that 
$$\alpha^* = \sum_i (\phi_i dg_i).$$
We now define the $TV_{\alpha^*}$ energy to be
$$TV_{\alpha^*}(df) := \sup_{\beta \in \Lambda^1(M), |\beta|\leq 1} \sum_i \int_{U_i} <g_i,d^*(\phi \beta)> + < f,d^* \beta> dV_g.$$
We remark that the lower semicontinuity property \ref{lsc} applies to the $TV_{\alpha^*}$ energy as well.
$$\text{If } f_k \xrightarrow{L^1} f \text{ then } TV_{\alpha^*}(df) \leq \liminf TV_{\alpha^*}(df_k). $$
This follows by applying the property to $g_i + f$ on each of the $U_i$. We also note that in the case $f \in W^{1,1}$, 
$$TV_{\alpha^*}(df) = \int_M |\alpha^* + df| dV_g.$$
We are now ready to restate and prove the second half of Theorem \ref{thm3}.
\begin{proposition} \label{gamma}
Let $\beta_t = \alpha^* + df_t$ be the solutions of the $t-GMS$ equation. Then for  every sequence $t_k \rightarrow \infty$, there exists a subsequence $t_{k_n}$ and a $TV_{\alpha^*}$ minimizer $f_{\infty}$ such that $f_{t_{k_n}} \xrightarrow{WBV} f_{\infty}$. 
\end{proposition}

\begin{proof}
By property \ref{compactness} of BV function, $\exists t_{k_n}$ and $f_{\infty}$ such that $f_{t_{k_n}} \xrightarrow{L^1} f_{\infty}$. It remains to show that $f_{\infty}$ is a minimizer of the $TV_{\alpha^*}$ energy and that $TV_{\alpha^*}(df_{\infty}) = \lim E^{t_{k_n}}_{GMS} (\beta_{t_{k_n}})$. We show this in two steps. 

First, we show that $TV_{\alpha^*} (df_{\infty}) \leq \liminf E^{t_{k_n}}_{GMS} (\beta_{t_{k_n}})$. We calculate
\begin{align*}
TV_{\alpha^*} (df_{\infty}) & \leq \liminf TV_{\alpha^*} (df_{t_{k_n}}) \\
& = \liminf \int_M |\alpha^* + df_{t_{k_n}}| dV_g \\ 
& = \liminf \int_M |\beta_{t_{k_n}}| dV_g \\ 
& \leq \liminf \int_M \sqrt{ t_{k_n}^{-2} + |\beta_{t_{k_n}}|^2} dV_g \\
& = \liminf E^{t_{k_n}}_{GMS} (\beta_{t_{k_n}})
\end{align*}
The first inequality holds by the lower semicontinuity of the $TV_{\alpha^*}$ energy, while the first equality holds by the definition of the $TV$ energy.

We now let $c = \liminf E^{t_{k_n}}_{GMS}(\beta_{t_{k_n}})$. We claim that $\inf_{f \in BV(M)} TV_{\alpha^*} (df) = c$, and thus $TV_{\alpha^*} (df_{\infty}) = c$ and $df_{\infty}$ is a minimizer of the $TV_{\alpha^*}$ energy. We prove by contradiction. 

Assume that $\exists g \in BV(M)$ and $\delta > 0$ such that $TV_{\alpha^*} (dg) = b < c - 3\delta$. By property \ref{density} of BV functions, we can find $g_k \in C^{\infty}$ such that $g_k \xrightarrow{WBV} g$. Thus, $TV_{\alpha^*}(dg_k) \rightarrow TV_{\alpha^*}(dg)$. We take $k$ large such that $TV_{\alpha^*} (dg_k) < c-2\delta$. We claim that for $t > \frac{\text{Vol}(M)}{\delta}$, $E^t_{GMS} (\alpha^* + dg_k) < c - \delta$. Indeed,
\begin{align*}
E^t_{GMS}(\alpha^* + dg_k) & = \int_M \sqrt{t^{-2} + |\alpha^* + dg_k|^2} dV_g \\
& \leq \int_M (t^{-1} + |\alpha^* + dg_k|) dV_g \\
& = \frac{\text{Vol}(M)}{t} + TV_{\alpha^*}(dg_k) \\
& < c - \delta
\end{align*}
On the other hand, $c = \liminf E^t_{GMS}(\alpha^* + df_t)$, so $E^t_{GMS} (\alpha^* + df_t) > c - \delta$ for $t$ large enough. As $\alpha^* + df_t$ minimizes the $E^t_{GMS}$ energy by definition, we have reached a contradiction. 
\end{proof}

\begin{remark*}
Our proof really shows that the $E^t_{GMS}(\alpha^* + df) \xrightarrow{\Gamma} TV_{\alpha^*}(df)$. $\Gamma$-convergence is typically shown by proving the "$\limsup$" and "$\liminf$" inequalities. (See \cite{braides}.) These are the two inequalities show above.
\end{remark*}

\section{Explicit Solutions on $S^k \times S^k$}\label{examples}
We will construct an explicit family of GMS $k$ forms on a metric conformal to the standard spherical metric on $S^k \times S^k$. For clarity, we will write $S^k_1 \times S^k_2$ to distinguish the two copies of $S^k$. These solutions will exhibit singularities in finite cohomology class for $k \geq 2$. We let $g_{E} = d\xi_{S^k_1}^2 + d\xi^2_{S^k_2}$ be the standard spherical metric on $S^k_1 \times S^k_2$. We will study $k$-forms which are solutions of the GMS equation with respect metrics which are in the conformal class $[g_E]$. In particular, we note that we can write $d\xi_{S^k_2}$ in spherical coordinates $d\theta^2 + \sin^2(\theta) d\xi^2_{S^{k-1}}$. We study metrics of the form $g_h = h^{-2}(\theta) g_E$ where $h(\theta)$ is a smooth positive function on $[0,\pi]$ which has a unique maximum at $\theta = 0$. We also require the compatibility condition that all of the odd derivatives $h^{(2k+1)}(0) = h^{(2k+1)}(\pi) = 0$ to ensure that $h$ is smooth at $\theta = 0$ and $\theta = \pi$. This condition follows from the fact that a radially symmetric function $h:\mathbb{R}^n \rightarrow \mathbb{R}$ is smooth at the origin iff its odd derivatives vanish at the origin.

We recall that the de Rham cohomology $H^2(S^k \times S^k) \cong \mathbb{R}^2$, where the cohomology classes are represented by $[\kappa_1 dV_{S^k_1} + \kappa_2 dV_{S^k_2}]$. We also note that we can calculate $\kappa_i$ by integrating over a submanifold homologous to a copy of $S^k_i \hookrightarrow S^k_1 \times S^k_2$. Setting $\omega_k$ as the volume of the unit sphere in Euclidean space, we find that
$$\kappa_1 = \frac{1}{\omega_k} \int_{S^k_1 \times \{0\}} \alpha \quad \text{ and } \quad 
\kappa_2 = \frac{1}{\omega_k} \int_{\{0\} \times S^k_2} \alpha.$$

We now consider the $k$-forms $\kappa dV_{S^k_2}$, multiples of the standard volume form on the second $S^k$ factor. We note that such forms are harmonic in the conformal class $[g_E]$ since harmonic $k$-forms are invariant under conformal change of metric in dimension $2k$. We now look for GMS solutions in the cohomology class $[\kappa dV_{S^k_2}]$.  
\begin{proposition}\label{cohomology}
Let $\alpha \in [\kappa dV_{S^k_2}]$ be a GMS solution. Then $\exists f(\theta)$ such that
$$\alpha = f(\theta) dV_{S^k_2}.$$
Furthermore,
$$\kappa = \frac{\omega_{k-1}}{\omega_k} \int_0^{\pi} f(\theta) \sin^{k-1}(\theta) d\theta.$$
\end{proposition} 
\begin{proof}
We begin by decomposing the space
$$\Lambda^k(M) = \bigoplus_{i=0}^k \Lambda^{k-i}(S^k_1) \wedge \Lambda^i(S^k_2).$$
We can then decompose $\alpha$ into its orthogonal components
$$\alpha = \sum_{i=0}^k \alpha_i \qquad \alpha_i \in \Lambda^{k-i}(S^k_1) \wedge \Lambda^i(S^k_2).$$
We first show that $\alpha_i = 0$ for $1 \leq i \leq k-1$. Assume that in local coordinates about a point $p$, 
$$\alpha_i = \sum_{1 \leq j_1 < \ldots j_i \leq k} \gamma_{j_1 \ldots j_i} \wedge dx^{j_1} \wedge \ldots \wedge dx^{j_i}.$$
Here, $\gamma_{j_1 \ldots j_i} \in \Lambda^{k-i}(S^k_1)$. As $\alpha_i \neq 0$, we know that for some $j_1 \ldots j_i$, $\gamma_{j_1 \ldots j_i}(p) \neq 0$. Choose an element $\phi \in SO(k)$ which fixes the point $p$ and such that $\phi^*(\gamma_{j_1 \ldots j_i})(p) \neq \gamma_{j_1 \ldots j_i}(p)$. (This can be done as long as $i \neq 0$.) Then
$$(\phi^*(\alpha)(p))_i = \phi^*(\alpha_i)(p) =  \sum_{1 \leq j_1 < \ldots j_i \leq k} \phi^*(\gamma_{j_1 \ldots j_i})(p) \wedge dx^{j_1} \wedge \ldots \wedge dx^{j_i}.$$
Thus, we find that $\phi^*(\alpha) \neq \alpha$. On the other hand, because $\phi$ is an isometry, $\phi^*(\alpha)$ is also a GMS solution. However, this contradicts the uniqueness of GMS solutions. Thus, $\alpha_i = 0$ for $1 \leq i \leq k-1$. Therefore, we obtain that $\exists f_1, f_2: M \rightarrow \mathbb{R}$ such that
$$\alpha = \alpha_0 + \alpha_k = f_1(x) dV_{S^k_1} + f_2(x) dV_{S^k_2}.$$ 
We now claim that for $i = 1,2$, $f_i$ depend only on the variable $\theta$. Indeed, if $p, q \in M$ are two points such that $\theta(p) = \theta(q)$ but $f_i(p) \neq f_i(q)$, then $\exists \phi$ an isometry of $M$ which fixes the level sets of $\theta$ but $\phi(p) = q$, which violates the uniqueness of GMS solutions, as described above. Thus,
$$\alpha = f_1(\theta)dV_{S^k_1} + f_2(\theta)dV_{S^k_2}.$$
We now apply the fact that $\alpha$ is closed.
$$ d\alpha = \frac{\partial f_1(\theta)}{\partial \theta}  d\theta \wedge dV_{S^k_1} = 0.$$
Thus, $\exists \kappa_1$ such that $f_1 \equiv \kappa_1$ and
$$\alpha = \kappa_1 dV_{S^k_1} + f_2(\theta) dV_{S^k_2}.$$
As $\alpha \in [\kappa dV_{S^k_2}]$, we find that $\kappa_1 = 0$ and 
\begin{align*}
\kappa & = \frac{1}{\omega_k} \int_{\{0\} \times S^k_2} f(\theta) dV_{S^k_2} \\
& = \frac{\omega_{k-1}}{\omega_k} \int_0^{\pi} f(\theta) \sin^{k-1}(\theta)d\theta.
\end{align*}
\end{proof}
We now calculate the actual GMS solution.
\begin{theorem} \label{explicit}
Consider the manifold $(S^k_1 \times S^k_2, g_h)$, using the notation described above. Let $c^* = h^{-k}(0)$ and let $c \in (-c^*,c^*)$ be a constant. Let 
$$f_c(\theta) = \frac{c}{\sqrt{1-c^2h^{2k}(\theta)}},\quad \text{ and } \quad
\kappa_c = \frac{\omega_{k-1}}{\omega_k} \int_0^{\pi} f_c(\theta) \sin^{k-1}(\theta) d \theta.$$
Then the $k$-form
$$\alpha_c := f_c(\theta)dV_{S^k_2} \in [\kappa_c dV_{S^k_2}]$$
is the unique GMS solution in $[\kappa_c dV_{S^k_2}]$. Furthermore, let $\kappa^* = \lim_{c \rightarrow c^*} \kappa_c$. Then the cohomology class $[\kappa dV_{S^k_2}]$ has a GMS solution iff 
$$|\kappa| < \kappa^*.$$
\end{theorem}
In Figure \ref{fig:s1xs1}, we plot $|\alpha_c(\theta)|$ for several values of $c$ on the manifold $(S^1 \times S^1,g_h)$ in the case where $h^2(\theta) = 1 + \cos^2(\theta)$. We see that as $c$ grows to its maximal value of $c^* = \frac{1}{\sqrt{2}}$, the solution becomes singular. 
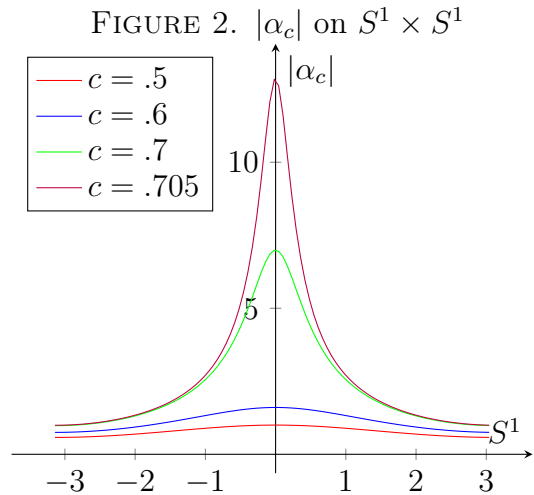
\begin{figure}[ht!]
\centering
\caption{$|\alpha_c|$ on $S^1 \times S^1$}
\begin{tikzpicture}
\begin{axis}[axis lines=middle,xlabel=$S^1$,ylabel=$|\alpha_c|$, enlargelimits,legend cell align={left}, legend pos = north west]
\addplot[domain=-pi:pi-.1,samples=100,red] {.5*sqrt(1+cos(deg(x)/2)^2)/sqrt(1-.5^2*(1+cos(deg(x)/2)^2)};
\addplot[domain=-pi:pi-.1,samples=100,blue] {.6*sqrt(1+cos(deg(x)/2)^2)/sqrt(1-.6^2*(1+cos(deg(x)/2)^2)};
\addplot[domain=-pi:pi-.1,samples=100,green] {.7*sqrt(1+cos(deg(x)/2)^2)/sqrt(1-.7^2*(1+cos(deg(x)/2)^2)};		
\addplot[domain=-pi:pi-.1,samples=100,purple] {.705*sqrt(1+cos(deg(x)/2)^2)/sqrt(1-.705^2*(1+cos(deg(x)/2)^2)};	
\legend{$c=.5$,$c=.6$,$c=.7$, $c=.705$}
\end{axis}
\end{tikzpicture}
\label{fig:s1xs1}
\end{figure}
\begin{proof} We note that by Proposition \ref{cohomology}, any GMS solution is necessarily of the form
$$f(\theta) dV_{S^k_2}.$$
We write down the GMS equation for such forms.
$$
d^* \left( \frac{f(\theta)}{\sqrt{1+h^{2k}(\theta)f^2(\theta)}} dV_{S^k_2} \right) = 0.
$$
However, because $dV_{S^k_2}$ is harmonic, we know that $d^*(dV_{S^k_2}) = 0$. Thus, we find that the equation reduces to the ODE
$$ \frac{d}{d \theta} \left( \frac{f(\theta)}{\sqrt{1+h^{2k}(\theta) f^2(\theta)}} \right) = 0.$$
We solve this ODE to find that $\exists c$ such that
$$f(\theta) = \frac{c}{\sqrt{1-c^2h^{2k}(\theta)}}.$$
Thus, every GMS solution is necessarily of the form $f_c(\theta) dV_{S^k_2}$. Additionally, from Proposition \ref{cohomology} we know that 
$$f_c(\theta) dV_{S^k_2} \in [\kappa_c dV_{S^k_2}].$$
From the definition of $f_c$, we see that $f_c$ is well defined iff $c^2 h^{2k}(\theta) < 1$ for all $\theta \in [0,\pi]$. As $h$ was chosen to attain its maximum at $\theta = 0$, we find that $f_c$ is well defined iff $c \in (-h^{-k}(0), h^{-k}(0)) = (-c^*, c^*)$. As $\kappa_c$ is a monotone, odd function of $c$, we find that the cohomology class has a solution iff 
$$|\kappa| < \lim_{c \rightarrow (c^*)^-} \kappa_c = |\kappa^*|.$$
\end{proof}
\begin{corollary}
For $k \geq 2$, there exist cohomology classes $[\kappa dV_{S^k_2}]$ with no GMS minimizer.
\end{corollary}
\begin{proof}
By Theorem \ref{explicit}, a cohomology class $[\kappa dV_{S^k_2}]$ has a GMS minimizer iff $|\kappa| < \kappa^*$. Thus, to find a cohomology class with no minimizer, it will suffice to find an example where $\kappa^* < \infty$. By definition,
$$\kappa^* = \lim_{c \rightarrow (c^*)^-} \int_0^\pi \frac{c}{\sqrt{1-(c)^2 h^{2k}(\theta)}} \sin^{k-1}(\theta) d\theta.$$
As $\kappa_c$ is monotone in $c$, the monotone convergence theorem tells us that 
$$\kappa^* = \int_0^\pi \frac{c^*}{\sqrt{1-(c^*)^2 h^{2k}(\theta)}} \sin^{k-1}(\theta) d\theta.$$
We now examine under what conditions this integral will be finite. We recall that in choosing $h$ we require that $h'(0) = 0$. However, we allowed $h''(0) \neq 0$. If we choose $h$ with $h''(0) \neq 0$, we observe that the function $\sqrt{1-(c^*)^2 h^{2k}(\theta)} \sim \theta$ around $\theta = 0$. We then find that the function 
$$ \frac{c^*}{\sqrt{1-(c^*)^2 h^{2k}(\theta)}} \sin^{k-1}(\theta) \sim \theta^{k-2}.$$
Thus, the integral
$$\int_0^\pi \frac{c^*}{\sqrt{1-(c^*)^2 h^{2k}(\theta)}} \sin^{k-1}(\theta) d\theta$$
is infinite in the case $k=1$, but finite in the case $k\geq 2$. This implies that for $k\geq 2$, we do not have a minimizer in every cohomology class.
\end{proof}
\section{Application to the Born Infeld Energy} \label{BI}
In this section, we discuss the relationship of the GMS energy to the Born Infeld energy in four dimensions. As mentioned in Section \ref{intro}, Born and Infeld introduced the Lagrangian
$$ E_{BI}(F) = \int_M \sqrt{\det (g-F)} dV_g.$$
When $M$ is a three dimensional Riemannian manifold, the Born Infeld energy reduces to the GMS energy. This can be seen by taking normal coordinates at a point. With $g = I$, we find
$$\det(I-F) = \det \left( \begin{array}{ccc}
1 & -F_{12} & -F_{13} \\ 
F_{12} & 1 & -F_{23} \\ 
F_{13} & F_{23} & 1
\end{array} \right) = 1 + |F|^2.$$
Thus the energy reduces to
$$E_{BI}(F) = \int_M \sqrt{1+|F|^2} dV_g.$$
which is the GMS energy. In four dimensions, the Born Infeld energy is more complicated. However, in special cases it reduces the GMS energy, as we will discuss.  
\begin{proposition}
Given $(M^4,g)$, $F \in \Lambda^2(M)$, the Born Infeld energy is given by
$$E_{BI}(F) = \int_M \sqrt{1+|F|^2 + \frac{1}{4}|F \wedge F|^2}dV_g.$$
\end{proposition}
\begin{proof}
Once again, we take normal coordinates about a point. We calculate
\begin{align*}
 \det \left( \begin{array}{cccc}
1 & -F_{12} & -F_{13} & -F_{14} \\ 
F_{12} & 1 & -F_{23} & -F_{24} \\ 
F_{13} & F_{23} & 1 & -F_{34} \\ 
F_{14} & F_{24} & F_{34} & 1
\end{array} \right)
& =  1 + |F|^2 + \left< 
\left( \begin{array}{c}
F_{12} \\ 
F_{13} \\ 
F_{14}
\end{array} \right),  
\left( \begin{array}{c}
F_{34} \\ 
-F_{24} \\ 
F_{23}
\end{array} \right)
\right>^2. \\
& = 1 + |F|^2 + \frac{1}{4} |F \wedge F|^2.
\end{align*}
\end{proof}
\begin{corollary}\label{energy}
Let $E_{BI}$ be the Born Infeld energy, $E_H$ be the standard Hodge energy, and $E_{GMS}$ be the generalized minimal surface energy defined in Section \ref{intro}. Then for all $F \in \Lambda^2(M)$,
$$E_{GMS}(F) \leq E_{BI}(F) \leq \text{Vol}(M) + \frac{1}{2}E_H(F).$$
Additionally, if $F$ is self-dual or anti-self-dual, $E_{BI}(F) = \text{Vol}(M) + \frac{1}{2}E_H(F)$. On the other hand, $E_{GMS}(F) = E_{BI}(F)$ iff $|F_+| \equiv |F_-|$.
\end{corollary}
\begin{proof}
Because 
$$ 0 \leq |F \wedge F|^2 \leq |F|^4 $$
we obtain that  
$$\int_M \sqrt{1+|F|^2} dV_g \leq \int_M \sqrt{1+|F|^2 + \frac{1}{4} |F \wedge F|^2} dV_g \leq \int_M (1 + \frac{1}{2}|F|^2) dV_g.$$
The first inequality becomes equality iff $|F \wedge F| = \big||F_+|^2 - |F_-|^2\big| \equiv 0$, or if $|F_+| \equiv |F_-|$. The second inequality become equality if $|F \wedge F| = \big| |F_+|^2 - |F_-|^2 \big| \equiv |F|^2$. Clearly, any self-dual or anti-self-dual $F$ satisfy this condition. 
\end{proof}
We now calculate the Euler Lagrange equation for the Born Infeld energy. 
\begin{proposition}
$F$ is a critical point of the Born Infeld energy in four dimensions iff 
$$ d^* \left( \sqrt{\det(I - g^{-1}F)}g (I - (g^{-1}F)^2)^{-1} g^{-1} F \right) = 0.$$
\end{proposition} 
\begin{proof}
We take a variation of the form $F(t) = F_0 + tdA$. 
\begin{align*}
\frac{d}{dt} \bigg|_{t=0} \int_M \sqrt{\det(g-F)} 
& = \frac{d}{dt} \bigg|_{t=0} \int_M \sqrt{\det(I-g^{-1}F)} dV_g \\
& = \int_M \frac{\frac{d}{dt}|_{t=0} \det(I-g^{-1}F)}{\sqrt{\det(I-g^{-1}F_0)}}dV_g \\
& = \int_M \frac{\det(I-g^{-1}F_0)\frac{d}{dt}|_{t=0} \det((I-g^{-1}F_0)^{-1}g^{-1}tdA)}{\sqrt{\det(I-g^{-1}F_0)}}dV_g \\
& = \int_M \sqrt{\det(I-g^{-1}F_0)} \det((I-g^{-1}F_0)^{-1} g^{-1}dA) dV_g \\
& = \left< \sqrt{\det(I-g^{-1}F_0)} g(I - g^{-1}F_0)^{-1}, dA \right>_{L^2}.
\end{align*}
Now let $\tilde{F} = \sqrt{\det(I-g^{-1}F_0)} g(I - g^{-1}F_0)^{-1}$. We can decompose $\tilde{F}$ into its symmetric and antisymmetric components, $\tilde{F} = \tilde{F}_{Sym} + \tilde{F}_{aSym}$. We now note that because $dA$ is antisymmetric, $\left< \tilde{F}_{Sym}, dA \right> \equiv 0$. Thus, the expression reduces to $\left< \tilde{F}_{aSym}, dA \right> = 0$. We now claim that $\forall t$
$$\left( (I-tg^{-1}F_0)^{-1} \right)_{aSym} = \left( I - (tg^{-1}F_0)^2\right)^{-1}g^{-1}F_0.$$
As this is a pointwise calculation, we reduce to normal coordinates, setting $g^{-1} = I$. Then, when $t \in [0,\epsilon]$ for small enough $\epsilon$, we can expand the LHS as a series
$$(I-tF_0)^{-1} = \sum_{n=0}^{\infty} \left( tF_0 \right)^n.$$
As $F_0$ is antisymmetric, $F_0^n$ is symmetric for even $n$ and antisymmetric for odd $n$. Thus,
\begin{align*}
\left( (I-tF_0)^{-1} \right)_{aSym} & = \sum_{n=0}^{\infty} \left( tF_0 \right)^{2n+1} \\
& = \sum_{n=0}^{\infty} \left( tF_0 \right)^{2n} tF_0 \\
& = \left( I - (tF_0)^2 \right)^{-1} tF_0.
\end{align*}
For each component of these matrices, this equation gives us a polynomial equation in $t$. As this identity holds on the interval $[0,\epsilon]$, the two polynomials must be equal and the identity necessarily holds for all $t \in \mathbb{R}$. In particular, for $t=1$, we obtain the identity
$$\left( (I-g^{-1}F_0)^{-1} \right)_{aSym} = \left( I - (g^{-1}F_0)^2\right)^{-1}g^{-1}F_0.$$
Returning to our Euler Lagrange equation, we find
\begin{align*} 
\frac{d}{dt} \bigg|_{t=0} \int_M \sqrt{\det(g-F)} 
& = \left< \sqrt{\det(I-g^{-1}F_0)}g \left(I - (g^{-1} F_0)^2\right)^{-1} g^{-1} F_0 , dA \right>_{L^2} \\
& = \left< d^* \left( \sqrt{\det(I-g^{-1}F_0)} g\left(I - (g^{-1} F_0)^2\right)^{-1} g^{-1}F_0 \right), A \right>_{L^2}.
\end{align*}
As $A$ is arbitrary, we obtain the equation
$$d^* \left( \sqrt{\det(I - g^{-1}F_0)} g \left(I - (g^{-1} F_0)^2\right)^{-1}g^{-1} F_0 \right) = 0.$$
\end{proof}
We now turn to the question of the existence of Born Infeld solutions in every cohomology class. We first show that self-dual and anti-self-dual $F$ are always Born Infeld solutions. On the other hand, we use our results from Section \ref{examples} to show that there exist cohomology classes where the Born Infeld solution becomes singular. 
\begin{theorem}
Every self-dual or anti-self-dual $F$ is a solution of the Born Infeld equations. 
\end{theorem}
\begin{proof}
We claim that in the (anti-)self-dual case,
$$\sqrt{\det(I-g^{-1}F)}g(I-(g^{-1}F)^2)^{-1}g^{-1}F = F.$$
We will study the expression at a point in normal coordinates. Our expression reduces to $\sqrt{\det(I-F)}(I-F^2)^{-1}F$. We begin by noting that in the (anti-)self-dual case, as described above in the proof of Corollary \ref{energy}, 
$$\sqrt{\det(I - F)} \equiv 1 + \frac{1}{2}|F|^2.$$
We now claim that in the (anti-)self-dual case,
$$ \left( I - F^2 \right)^{-1} = \frac{1}{1+\frac{1}{2}|F|^2} I.$$
We note that by the spectral theorem for antisymmetric matrices, we can pick  special coordinates about $p$ such that $F(p) = F_{12} dx^1 \wedge dx^2 + F_{34} dx^3 \wedge dx^4$. Then
$$ I - F^2 = \left( \begin{array}{cccc}
1+F_{12}^2 & 0 & 0 & 0 \\ 
0 & 1+F_{12}^2 & 0 & 0 \\ 
0 & 0 & 1+F_{34}^2 & 0 \\ 
0 & 0 & 0 & 1+F_{34}^2
\end{array} \right) .$$
The (anti-)self-dual condition tells us that $F_{12}^2 = F_{34}^2$. So $1+F_{12}^2 = 1+F_{34}^2 = 1+ \frac{1}{2} |F|^2$. Thus, we find that 
$$(I-F^2)^{-1} = \frac{1}{1 + \frac{1}{2}|F|^2} I.$$
Therefore, in the (anti-)self-dual case, the Born Infeld equation reduces to
$$d^* \left( \sqrt{\det(I - g^{-1}F_0)} g \left(I - (g^{-1} F_0)^2\right)^{-1} g^{-1} F_0 \right) = d^*F =0.$$
Thus, (anti-)self-dual forms, which are necessarily harmonic, satisfy the Born Infeld equation.
\end{proof}
We next turn to an example of a cohomology class where the Born Infeld solution is singular. We use the notation of Section \ref{examples}. 
\begin{theorem}
Let $M = S^2 \times S^2$ with metric $g$ given by $g = h^{-2}(\theta)g_E$ as defined in Section \ref{examples}. For $|c| \leq c^* = h^{-2}(0)$, let
$$F^c = \frac{c}{\sqrt{1-c^2h^{4}(\theta)}} dV_{S^2_2}, \quad \text{ and } 
\quad \kappa_c = \frac{1}{2\pi} \int_0^{\pi} \frac{c}{\sqrt{1-c^2h^{4}(\theta)}} \sin(\theta) d \theta.$$
Then, when $|c| < c^*$, $F^c$ is a minimum of the Born Infeld energy in its cohomology class $[\kappa_c dV_{S^2_2}]$. Furthermore, if $h$ is chosen so that $\kappa_{c^*} < \infty$, the cohomology class $[\kappa_{c^*} dV_g]$ has no smooth Born Infeld minimizer.
\end{theorem}
\begin{proof}
We begin with the case $|c| < c^*$. We decompose $F^c$ into its self-dual and anti-self-dual parts.
$$\frac{c}{\sqrt{1-c^2h^{4}(\theta)}} dV_{S^2_2} = \frac{c}{\sqrt{1-c^2h^{4}(\theta)}} \left[ \frac{1}{2}(dV_{S^2_1} + dV_{S^2_2}) - \frac{1}{2}(dV_{S^2_1} - dV_{S^2_2}) \right]. $$
We note that $|F^c_+| \equiv |F^c_-|$. Thus, by Corollary \ref{energy}, we find that $E_{GMS}(F^c) = E_{BI}(F^c)$. As $F^c$ minimizes $E_{GMS}$ and $E_{GMS} \leq E_{BI}$, we find that $F^c$ minimizes $E_{BI}$.

We now study the cohomology class $[\kappa_{c^*} dV_g]$ in the case $\kappa_{c^*} < \infty$. In this case, $F^{c^*}$ is a singular two form. However, both the GMS and Born Infeld energies are finite and are given by
$$E_{BI}(F^{c^*}) = E_{GMS}(F^{c^*}) = \int_M \sqrt{1+|F^{c^*}|^2} dV_g.$$
Through the rest of the section, $[\cdot]$ will refer only to the smooth forms in the given cohomology class. We claim that
$$ E_{BI}(F^{c^*}) = \inf_{F \in [\kappa_{c^*} dV_g]} E_{BI}(F).$$
We first show that 
$$E_{BI}(F^{c^*}) \leq \inf_{F \in [\kappa_{c^*} dV_g]} E_{BI}(F).$$
We prove by contradiction. Assume that $\exists F \in [\kappa_{c^*} dV_g]$ such that 
$$E_{BI}(F) < E_{BI}(F^{c^*}) - \delta.$$
Then, setting $\mu_c = \frac{c}{c^*}$ for $c < c^*$,
\begin{align*}
E_{BI}(\mu_c F) & = \int_M \sqrt{ 1 + \mu_c^2 |F|^2 + \mu_c^4 \frac{1}{4} |F \wedge F|^2 } dV_g \\
& < E_{BI}(F) \\
& < E_{BI}(F^{c^*}) - \delta.
\end{align*}
On the other hand, we note that by the dominated convergence theorem, as $c \rightarrow c^*$, 
$$E_{BI}(F^c) \rightarrow E_{BI}(F^{c^*}).$$
Thus, for $c^* - c$ small enough, $E_{BI}(F^c) > E_{BI}(F^{c^*}) - \delta > E_{BI} (\mu_c F)$. However, $E_{BI}(F^c)$ was shown to minimize $E_{BI}$ in the $[\kappa_c dV_g]$ cohomology class, a contradiction. 

We next show that 
$$E_{BI}(F^{c^*}) \geq \inf_{F \in [\kappa_{c^*}dV_g]} E_{BI}(F).$$
Once again setting $\mu_c = \frac{c}{c^*}$, we consider the sequence $\mu_c^{-1}F^c \in [\kappa_{c^*} dV_g]$. Then
$$ E_{BI}(F^c) \leq E_{BI}(\mu_c^{-1} F^c) \leq \mu_c^{-1} E_{BI}(F^c),$$
following from the definitions,
$$\int_M \sqrt{1+|F^c|^2}dV_g \leq \int_M \sqrt{1+\mu_c^{-2}|F^c|^2}dV_g \leq \int_M \mu_c^{-1} \sqrt{1+|F^c|^2}dV_g.$$
Because $E_{BI}(F^c) \rightarrow E_{BI}(F^{c^*})$ and $\mu_c^{-1} \rightarrow 1$ as $c \rightarrow c^*$, we find that as $c \rightarrow c^*$, 
$$ E_{BI}(\mu_c^{-1}F^c) \rightarrow E_{BI}(F^{c^*}).$$
Thus, the forms $\mu_c^{-1}F^c \in [\kappa_{c^*}dV_g]$ demonstrate that $E_{BI}(F^{c^*})$ is an upper bound on $\inf_{F \in [\kappa_{c^*}dV_g]} E_{BI}(F).$ We conclude that
$$E_{BI}(F^{c^*}) = \inf_{F \in [\kappa_{c^*}dV_g]} E_{BI}(F).$$ 
We now note that a similar proof can be repeated for $E_{GMS}$ to obtain
$$E_{GMS}(F^{c^*}) = \inf_{F \in [\kappa_{c^*}dV_g]} E_{GMS}(F).$$
We next show that this energy is not attained by any smooth form. Assume that $\exists F \in [\kappa_{c^*}dV_g]$ such that 
$$E_{BI}(F) = E_{BI}(F^{c^*}).$$
Then
$$E_{GMS}(F) \leq E_{BI}(F) = \inf_{F \in [\kappa_{c^*} dV_g]} E_{GMS}(F).$$
Thus, $F$ minimizes the $E_{GMS}$ energy. However, Theorem \ref{examples} shows that no such minimizer can exist. We conclude that $E_{BI}$ has no minimizer in $[\kappa_{c^*} dV_g]$. 
\end{proof}

\bibliographystyle{plain}
\bibliography{minform}

\begin{thebibliography}{10}

\bibitem{ambrosio}
Luigi Ambrosio, Nicola Fusco, and Diego Pallara.
\newblock {\em Functions of bounded variation and free discontinuity problems},
  volume 254.
\newblock Clarendon Press Oxford, 2000.

\bibitem{bombieri}
Enrico Bombieri, Ennio De~Giorgi, and Mario Miranda.
\newblock Una maggiorazione a priori relativa alle ipersuperfici minimali non
  parametriche.
\newblock {\em Archive for Rational Mechanics and Analysis}, 32(4):255--267,
  1969.

\bibitem{born}
Max Born and Leopold Infeld.
\newblock Foundations of the new field theory.
\newblock {\em Proceedings of the Royal Society of London. Series A, Containing
  Papers of a Mathematical and Physical Character}, 144(852):425--451, 1934.

\bibitem{braides}
Andrea Braides.
\newblock {\em Gamma-convergence for Beginners}, volume~22.
\newblock Clarendon Press, 2002.

\bibitem{giusti}
Enrico Giusti.
\newblock Minimal surfaces and functions of bounded variation.
\newblock {\em Monogr. Math.}, 80, 1984.

\bibitem{sibnerExistence2}
Lesley~M Sibner.
\newblock An existence theorem for a non-regular variational problem.
\newblock {\em manuscripta mathematica}, 43(1):45--72, 1983.

\bibitem{sibnerExistence1}
LM~Sibner and RJ~Sibner.
\newblock A non-linear hodge-de rham theorem.
\newblock {\em Acta Mathematica}, 125(1):57--73, 1970.

\bibitem{sibnerExample1}
LM~Sibner and RJ~Sibner.
\newblock Nonlinear hodge theory: applications.
\newblock {\em Advances in Mathematics}, 31(1):1--15, 1979.

\bibitem{sibnerExample2}
LM~Sibner and RJ~Sibner.
\newblock Transonic flow on an axially symmetric torus.
\newblock {\em Journal of Mathematical Analysis and Applications},
  72(1):362--382, 1979.

\bibitem{simonEstimates}
Leon Simon.
\newblock Interior gradient bounds for non-uniformly elliptic equations.
\newblock {\em Indiana University Mathematics Journal}, 25(9):821--855, 1976.

\bibitem{simon}
Leon Simon et~al.
\newblock {\em Lectures on geometric measure theory}.
\newblock The Australian National University, Mathematical Sciences Institute,
  Centre for Mathematics \& its Applications, 1983.

\bibitem{spruck}
Joel Spruck.
\newblock Interior gradient estimates and existence theorems for constant mean
  curvature graphs in mn$\times$ r.
\newblock {\em Pure Appl. Math. Q}, 3(3):785--800, 2007.

\end{thebibliography}
\end{document}